\documentclass[12pt]{article}
 \usepackage{amsthm,amsmath,amssymb, pgf, tikz, pgfpict2e}

\theoremstyle{plain}
\newtheorem{Thm}{Theorem}
\newtheorem{Cor}[Thm]{Corollary}

\newtheorem{lemma}[Thm]{Lemma}
\newtheorem{Prop}[Thm]{Proposition}
\newtheorem{Def}[Thm]{Definition}

\newtheorem{remark}[Thm]{Remark}
\newtheorem{example}[Thm]{Example}
\newtheorem{Ques}{Question}[section]

\date{June 23, 2020   }
\title{The Distinguishing Number  and Distinguishing Chromatic Number for Posets}

\author{Karen L. Collins\\
\small Dept. of Mathematics and Computer Science\\
\small Wesleyan University\\
\small Middletown CT 06459-0128\\
\small\tt kcollins@wesleyan.edu\
\and
Ann N. Trenk\thanks{
This work was supported by a grant from the Simons Foundation (\#426725, Ann Trenk). } \\
\small Department of Mathematics\\
\small Wellesley College\\
\small Wellesley MA 02481\\
\small\tt atrenk@wellesley.edu
}

\begin{document}
\maketitle

\begin{abstract}     
In this paper we introduce  the concepts of the distinguishing number and the distinguishing chromatic number of a poset. For a  distributive lattice $L$  and its set  $Q_L$ of join-irreducibles, we use   classic lattice theory to show that any linear extension of $Q_L$ generates a distinguishing 2-coloring of $L$.   We prove general upper bounds for the distinguishing chromatic number and particular upper bounds for the Boolean lattice and for divisibility lattices.     In addition, we show that the distinguishing number of any twin-free Cohen-Macaulay planar lattice is at most  2.

  \end{abstract}

\bibliographystyle{plain} 


\section{Introduction}   
\label{sect-intro}
The distinguishing number of a graph,  introduced by Albertson and Collins   \cite{AlCo96}, is the smallest integer $k$  for which the vertices can be colored using $k$ colors so that the only automorphism of the graph that preserves colors is the identity.   The distinguishing chromatic number,  introduced by Collins and Trenk \cite{CoTr06}, has the additional requirement that the coloring of the vertices is proper, that is, adjacent vertices get different colors.   The distinguishing number of graph $G$ is denoted by $D(G)$ and the distinguishing chromatic number by $\chi_D(G)$.   These and related topics have received considerable attention by many authors in recent years; see, for example, \cite{AlSo18,
BaPa17-b,
Cr18,
ImSmTuWa15,
Le17,
Pi17-b}.
In this paper, we introduce the distinguishing number and the  distinguishing chromatic number of a poset.

There are several challenges in studying  these parameters.    A distinguishing coloring of a graph or poset does not always yield a distinguishing coloring of  induced subgraphs or subposets.  It is possible to have an graph $H$ induced in  graph $G$ for which $D(H) > D(G)$, and the same holds for posets.  We provide an example of this  following Definition~\ref{disting-def}.
  In addition, the structure inherent in posets makes these parameters qualitatively different from the graph versions.


We end this section with  an overview of the rest of the paper. In Section~\ref{sect-prelim}, we provide background material about posets, lattices, and distributive lattices.    We introduce the distinguishing number of a poset in Section~\ref{sect-dist-nos}, and prove results about sums of chains, distributive lattices, and divisibility lattices.  In Section~\ref{sect-dist-chr}, we study the distinguishing chromatic number, giving upper bounds for the distinguishing chromatic number of distributive lattices, divisibility lattices, and Boolean lattices.  We also show  that there exist posets $P$ for which the gap between the distinguishing chromatic number of $P$ and that of its comparability graph is arbitrarily large.  We return to the distinguishing number in Section~\ref{sect-planar} and focus on ranked planar lattices (equivalently, Cohen-Macaulay) that are rank-connected.  We conclude with  several open questions.


 \section{Preliminaries} \label{sect-prelim}
In this section we provide   definitions related to posets and lattices, and present Birkhoff's classic lattice theorem, which we use as a tool in Section~\ref{sect-dist-nos} and Section~\ref{sect-dist-chr}.  For additional details and background, see \cite{St12}. 

\subsection{General poset definitions}\label{sect-poset-defn}

The posets we consider are finite and reflexive.  If $P$ is the poset $(X,\preceq)$, we call $X$ the \emph{ground set} of $P$ and refer to the elements of $X$ (which we also call the elements of $P$)  as \emph{points}.    We write $x \prec y$ if $x \preceq y$ and $x \neq y$.   If $x \preceq y$ or $y \preceq x$, we say points $x$ and $y$ are \emph{comparable}, and otherwise they are \emph{incomparable}.   We say that $y $ \emph{covers}  $x$
if $x \prec y$ and there is no other point $v$ with $x \prec v \prec y$.    An \emph{automorphism} of poset $P = (X,\preceq)$ is a bijection from $X$ to $X$ that preserves the relation $\preceq$.

 A set of pairwise comparable points in a poset   is called a \emph{chain}, and if the points are pairwise incomparable they form an \emph{antichain}.  An $r$-chain is a chain with $r$  points, and such a chain has  \emph{length} $r-1$.  The \emph{height} of a poset is the size of a  maximum chain and the \emph{width} is the size of a  maximum antichain.  
 
 If a poset has a unique minimal element, we call this element $\hat{0}$ and if it has a unique maximal element we call it $\hat{1}$.  We say that a poset with a $\hat{0}$ and $\hat{1}$ is \emph{ranked} if every maximal chain from $\hat{0}$  to $\hat{1}$ has the same length.  The \emph{rank} of a point $x$  in a poset, denoted by $rank(x)$ or $r(x)$,  is the length of a longest chain that has  $x$ as its largest element. For example, in Figure~\ref{ex-fig}, each of posets  $L_{pq^2}, L_{p^2q^2}, M$ and $L_{pqr}$  has a $\hat{0}$ and a $\hat{1}$, while poset $S_4$ has neither, and
 $L_{pq^2}, L_{p^2q^2}$ and $L_{pqr}$ are ranked, while poset $M$ is not. 
   
 A poset is \emph{planar}  if its Hasse diagram can be drawn in the plane with  no edges crossing and so that the edge from $a$ to $b$ has strictly increasing $y$-coordinate when $a \prec b$.   In Figure~\ref{ex-fig},  $L_{pq^2}, L_{p^2q^2}, M$ and $S_4$ are planar,  even though the drawing shown  of $S_4$  has edges crossing. 
We demonstrate in Section \ref{sect-planar} that poset $L_{pqr}$ is not planar (see Remark~\ref{B-not-planar}).

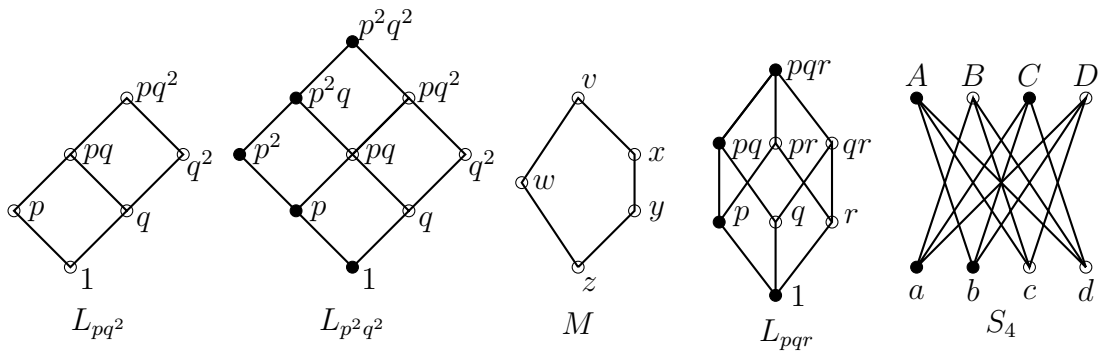
\begin{figure}[h]
\begin{center}
\begin{tikzpicture}[scale=.75]
\draw[black]
(0,0) circle [radius=3pt]
(-1,1) circle [radius=3pt]
(0,2) circle [radius=3pt]
(1,3) circle [radius=3pt]
;
\draw
(1,1) circle [radius=3pt]
(2,2) circle [radius=3pt]

;

\draw[thick] (0,2)--(1,3)--(2,2)--(1,1)--(0,2)--(-1,1) -- (0,0)--(1,1);
\node(0) at (.5, -1) {$L_{pq^2}$};
\node(0) at (.3,-.2) {$1$};
\node(0) at (1.3,.8) {$q$};
\node(0) at (2.3,1.8) {$q^2$};
\node(0) at (-.6,1) {$p$};
\node(0) at (.5,2) {$pq$};
\node(0) at (1.55,3.2) {$pq^2$};

\filldraw[black]
(3,2) circle [radius=3pt]
(4,3) circle [radius=3pt]
(5,4) circle [radius=3pt]
(5,0) circle [radius=3pt]
(4,1) circle [radius=3pt]
;
\draw

(6,1) circle [radius=3pt]
(5,2) circle [radius=3pt]
 (7,2) circle [radius=3pt]
(6,3) circle [radius=3pt]
;

\draw[thick] (5,2)--(6,3)--(7,2)--(6,1)--(5,2)--(4,1)-- (5,0)--(6,1);
\draw[thick] (4,1)--(3,2)--(4,3)--(5,4)--(6,3)--(5,2)--(4,3);

\node(0) at (5, -1) {$L_{p^2q^2}$};
\node(0) at (5.3,-.2) {$1$};
\node(0) at (6.3,.8) {$q$};
\node(0) at (7.3,1.8) {$q^2$};
\node(0) at (4.4,1) {$p$};
\node(0) at (5.5,2) {$pq$};
\node(0) at (6.55,3.2) {$pq^2$};
\node(0) at (3.5,2.1) {$p^2$};
\node(0) at (4.6,3.1) {$p^2q$};
\node(0) at (5.6,4.3) {$p^2q^2$};

\draw
(9,0) circle [radius=3pt]
(8,1.5) circle [radius=3pt]
(9,3) circle [radius=3pt]
(10,2) circle [radius=3pt]
(10,1) circle [radius=3pt]
;

\draw[thick] (9,0)--(8,1.5)--(9,3)--(10,2)--(10,1)--(9,0) ;
\node(0) at (9, -1) {$M$};
\node(0) at (9.2,-.3) {$z$};
\node(0) at (10.4,1) {$y$};
\node(0) at (10.4,2) {$x$};
\node(0) at (8.4,1.5) {$w$};
\node(0) at (9.2,3.3) {$v$};

\filldraw[black]
(11.5,.8) circle [radius=3pt]
(11.5,2.2) circle [radius=3pt]
(12.5,-.5) circle [radius=3pt]
(12.5,3.5) circle [radius=3pt]

;
\draw

(12.5,.8) circle [radius=3pt]
(12.5,2.2) circle [radius=3pt]
(13.5,.8) circle [radius=3pt]
(13.5,2.2) circle [radius=3pt]
;

\draw[thick] (13.5,2.2)--(12.5,3.5)--(11.5,2.2)--(12.5,.8);
\draw[thick](12.5,2.2)--(12.5,3.5)--(11.5,2.2)--(11.5,.8)--(12.5,-.5)--(12.5,.8)--(13.5,2.2)--(13.5,.8)--(12.5,-.5);
\draw[thick] (11.5,.8)--(12.5,2.2)--(13.5,.8);
\node(0) at (12.7, -1.2) {$L_{pqr}$};

\node(0) at (12.9,-.5) {$1$};
\node(0) at (11.9,.9) {$p$};
\node(0) at (12.9,.9) {$q$};
\node(0) at (13.85,.9) {$r$};
\node(0) at (12,2.1) {$pq$};
\node(0) at (13,2.1) {$pr$};
\node(0) at (13.95,2.1) {$qr$};
\node(0) at (13.1,3.5) {$pqr$};
\filldraw[black]
(15,0) circle [radius=3pt]
(16,0) circle [radius=3pt]
 (15,3) circle [radius=3pt]
(17,3) circle [radius=3pt]
;
\draw
 (17,0) circle [radius=3pt]
 (18,0) circle [radius=3pt]
 (16,3) circle [radius=3pt]
(18,3) circle [radius=3pt]
;

\draw[thick] (15,0)--(16,3)--(17,0)--(18,3)--(16,0)--(15,3)--(18,0)--(16,3);
\draw[thick](18,0)--(17,3)--(15,0)--(18,3);
\draw[thick] (17,3)--(16,0);
\draw[thick] (15,3)--(17,0);
\node(0) at (16.5, -1) {$S_4$};
\node(0) at (15,-.45) {$a$};
\node(0) at (16,-.4) {$b$};
\node(0) at (17,-.45) {$c$};
\node(0) at (18,-.4) {$d$};
\node(0) at (15,3.4) {$A$};
\node(0) at (16,3.4) {$B$};
\node(0) at (17,3.4) {$C$};
\node(0) at (18,3.4) {$D$};
\end{tikzpicture}

\caption{Examples of posets with distinguishing labelings ($p$, $q$ and $r$ are distinct primes).}
\label{ex-fig}
\end{center}
\end{figure}

\vspace{-.3in}

\subsection{Lattice definitions}

A point $z $ in poset $P$ is called the  \emph{meet} of $x$ and $y$ in $P$,  and denoted by $x\wedge y$, if it is   the unique largest element  in $ P$ such that $z\preceq x$ and $z\preceq y$. Thus,  if   $x\wedge y$ exists and  $a\preceq x$ and $a\preceq y$, then $a\preceq  x\wedge y$.     Similarly, a point $w \in P$ is called the \emph{join} of $x$ and $y$ in $P$, and denoted by $x\vee y$, if it is the unique smallest element $w\in P$ such that $w\succeq x$ and $w\succeq y$.  Thus,    if   $x\vee y$ exists  and $a\succeq x$ and $a\succeq y$, then $a\succeq x\vee y$.   

A poset $L$ is a \emph{lattice} if    $x\wedge y$ and $x\vee y$ both exist for all points $x$ and $y$ in L.  Furthermore, $L$ is a \emph{distributive lattice} if $\wedge$ and $\vee$ satisfy the distributive laws
\[(x\wedge y)\vee z = (x\vee z)\wedge (y\vee z)\]
\[(x\vee y)\wedge z = (x\wedge z)\vee (y\wedge z)\]
for all $x,y,z\in L$.
For example, all the posets in Figure~\ref{ex-fig} are lattices except for $S_4$, and all the lattices are distributive except for $M$. A lattice need not be ranked ($M$ is not ranked), but a distributive lattice is ranked, as  a consequence of Birkhoff's Theorem (Theorem~\ref{birkhoff-thm} in Section~\ref{sect-birktheorem}).  
 
A point $x$ in a lattice is called \emph{join-irreducible} if in the Hasse diagram of the lattice $x$ has exactly one downward edge.
 For example, in Figure~\ref{ex-fig}, the join-irreducible points of $L_{pq^2}$ are $p,q,$ and $q^2$, while there are no join-irreducible points in  poset $S_4$. 
As we will see in Birkhoff's Theorem, the join-irreducible points of a distributive lattice generate all the elements in the lattice by the join operation. In this way, they act like the prime numbers in the prime factorization of an integer.  

\subsection{Birkhoff's Theorem} \label{sect-birktheorem}

In this section, we present a fundamental  theorem due to Birkhoff (Theorem~\ref{birkhoff-thm}) and a corollary, both of which will be used as tools in later sections of the paper.
Let poset $P= (X,\preceq)$.
 The \emph{downset} of a point $a \in X$ is defined as $down(a) = \{x\in X: x \preceq a \}$ and the downset of a subset $A \subseteq X$ is defined as
$down(A) = \{x \in X: x \preceq a \hbox{ for some } a \in A\}$.
\begin{Def}{\rm   Let $P= (X,\preceq)$ be a poset.  The  \emph{downset lattice\/}  $J(P)$ has ground set $\{down(S): S\subseteq X\}$ and the relation is $\subseteq$.
}
\end{Def}

Observe that if $P$ is a poset,  then $J(P)$ is a distributive lattice,  in which the meet of elements $S$ and $T$ is $S \cap T$ and the join of these elements is $S \cup T$.

\begin{example} {\rm
 In Figure~\ref{L-150-fig}, the join-irreducible elements of $ L_{150}$ are $a = 2$, $b=3$, $c=5$,  and $d=25$.    When these are ordered using the ordering induced by $L_{150}$ they produce the poset labeled $Q_L$   also shown   in Figure~\ref{L-150-fig}.  There are 16 subsets of elements of $Q_L$,  producing 12 distinct downsets, which  are  given in the following table.  When these 12 downsets are ordered by set inclusion, we obtain the downset lattice $J(Q_L)$ which is isomorphic to the original lattice $L_{150}$.  }
\small
\begin{center}
\begin{tabular}{|c||c|c|c|c|c|c|c|c|c|c|c|c|c|c|c|c|}\hline
$S$ &  $\emptyset$ & $a$ &$ b$ & $c$ & $d$ & $ab$ & $ac$ & $ad$ & $bc$ &$ bd$ & $cd$ & $abc$ & $abd$ & $acd$ & $bcd$ & $abcd$ \\ \hline
$down(S)$& $\emptyset $ & $a$ & $b$ & $c$ & $cd$ & $ab$ & $ac$ & $acd$  & $bc$ & $bcd$ & $cd$ & $abc$ & $abcd$ & $acd$ & $bcd$ & $abcd$ \\ \hline
\end{tabular}       
\end{center}
\normalsize 

\label{JP2-example}
\end{example}

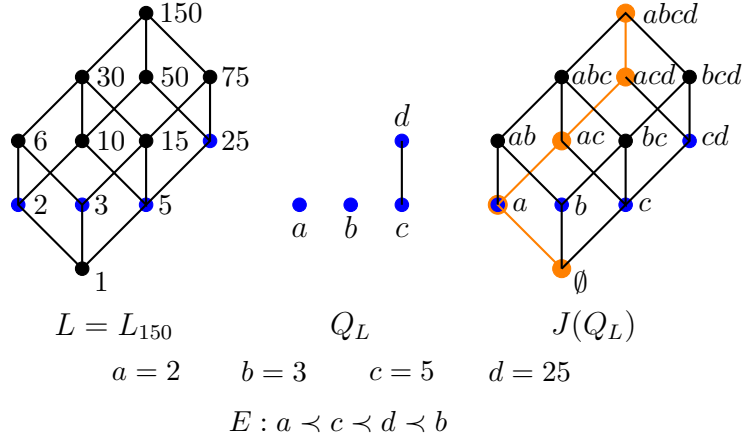
\begin{figure}[h]
\begin{center}
\begin{tikzpicture}[scale=.85]
\filldraw[black]
(1,0) circle [radius=3pt]
(0,2) circle [radius=3pt]
(1,2) circle [radius=3pt]
(2,2) circle [radius=3pt]
(1,3) circle [radius=3pt]
(2,3) circle [radius=3pt]
(3,3) circle [radius=3pt]
(2,4) circle [radius=3pt]

;
\filldraw[blue]
(0,1) circle [radius=3pt]
(1,1) circle [radius=3pt]
(2,1) circle [radius=3pt]
(3,2) circle [radius=3pt]
;

\draw[thick] (1,0)--(0,1)--(0,2)--(1,3)--(2,4)--(3,3)--(3,2)--(2,1)--(1,0);
\draw[thick] (1,0)--(1,1)--(2,2)--(3,3);
\draw[thick] (2,4)--(2,3)--(1,2)--(0,1);
\draw[thick] (1,1)--(0,2);
\draw[thick] (2,2)-- (2,1)--(1,2);
\draw[thick] (1,2)--(1,3)--(2,2);
\draw[thick] (2,3)--(3,2);
\node(0) at (1.5,-.9) {$L = L_{150}$};
\node(0) at (1.3,-.2) {\small $1$};
\node(0) at (.35,1) {\small $2$};
\node(0) at (1.3,1) {\small $3$};
\node(0) at (2.3,1) {\small $5$};
\node(0) at (0.35,2) {\small $6$};
\node(0) at (1.45,2) {\small $10$};
\node(0) at (2.45,2) {\small $15$};
\node(0) at (3.4,2) {\small $25$};
\node(0) at (1.45,3) {\small $30$};
\node(0) at (2.45,3) {\small $50$};
\node(0) at (3.4,3) {\small $75$};
\node(0) at (2.55,4) {\small $150$};

\filldraw[blue]
(4.4,1) circle [radius=3pt]
(5.2,1) circle [radius=3pt]
(6,1) circle [radius=3pt]
(6,2) circle [radius=3pt]
;

\draw[thick] (6,1)--(6,2);
\node(0) at (5.2,-.9) {$Q_L$};
\node(0) at (4.4,.6) {$a$};
\node(0) at (5.2,.65) {$b$};
\node(0) at (6,.6) {$c$};
\node(0) at (6,2.4) {$d$};

\filldraw[black]
(8.5,0) circle [radius=3pt]
(7.5,2) circle [radius=3pt]
(8.5,2) circle [radius=3pt]
(9.5,2) circle [radius=3pt]
(8.5,3) circle [radius=3pt]
(9.5,3) circle [radius=3pt]
(10.5,3) circle [radius=3pt]
(9.5,4) circle [radius=3pt]

;
\filldraw[orange] 
(8.5,0) circle [radius=4pt]
(7.5,1) circle [radius=4.2pt]
(8.5,2) circle [radius=4pt]
(9.5,3) circle [radius=4pt]
(9.5,4) circle [radius=4pt]
;
\filldraw[blue]
(7.5,1) circle [radius=3pt]
(8.5,1) circle [radius=3pt]
(9.5,1) circle [radius=3pt]
(10.5,2) circle [radius=3pt]
;
\draw[thick, orange] (8.5,0)--(7.5,1)--(8.5,2)--(9.5,3)--(9.5,4);
\draw[thick] (7.5,1)--(7.5,2)--(8.5,3)--(9.5,4)--(10.5,3)--(10.5,2)--(9.5,1)--(8.5,0);
\draw[thick] (8.5,0)--(8.5,1)--(9.5,2)--(10.5,3);
\draw[thick] (8.5,1)--(7.5,2);
\draw[thick] (9.5,2)-- (9.5,1)--(8.5,2);
\draw[thick] (8.5,2)--(8.5,3)--(9.5,2);
\draw[thick] (9.5,3)--(10.5,2);
\node(0) at (9,-.9) {$J(Q_L)$};
\node(0) at (8.8,-.2) {\small $\emptyset$};
\node(0) at (7.85,1) {\small $a$};
\node(0) at (8.8,1) {\small $b$};
\node(0) at (9.8,1) {\small $c$};
\node(0) at (7.88,2.1) {\small $ab$};
\node(0) at (8.95,2.05) {\small $ac$};
\node(0) at (9.95,2.08) {\small $bc$};
\node(0) at (10.9,2.1) {\small $cd$};
\node(0) at (8.97,3.05) {\small $abc$};
\node(0) at (9.95,3.05) {\small $acd$};
\node(0) at (11,3.05) {\small $bcd$};
\node(0) at (10.2,4) {\small $abcd$};

\node(0) at (2,-1.6) {\small $a=2$};
\node(0) at (4,-1.6) {\small $b=3$};
\node(0) at (6,-1.6) {\small $c=5$};
\node(0) at (8,-1.6) {\small $d=25$};
\node(0) at (5,-2.4) {\small $E: a\prec c \prec d\prec b$};
\end{tikzpicture}

\end{center}

\caption{The lattice $L_{150}$,  its poset $Q_L$  of join-irreducibles,  a linear extension $E$ of $Q_L$, and the downset lattice $J(Q_L)$, together with a distinguishing coloring of $J(Q_L)$.}
\label{L-150-fig}

\end{figure}

This example illustrates the following classic theorem due to Birkhoff \cite{Bi33} and called the \emph{Fundamental Theorem of Distributive Lattices} in \cite{St12}.

\begin{Thm} If $L $ is a distributive lattice and $Q_L$  is the  poset  induced by the join-irreducible points of $L$, then $J(Q_L)$ is isomorphic to $L$.
Indeed,   the  function $f:L \to J(Q_L)$ defined by $f(w) = \{y \in Q_L: y \preceq w \}$ is an isomorphism. 

\label{birkhoff-thm}
\end{Thm}

The following notation will be helpful as we use Theorem~\ref{birkhoff-thm} repeatedly.

\begin{Def} {\rm For a distributive lattice $L$, denote by $Q_L$ the induced poset of all join-irreducible points of $L$.  }

\end{Def}

Birkhoff's theorem  is fundamental in several ways. First it provides a method for checking whether a poset  $L$ is a distributive lattice without having to verify that every pair of points has a meet and a join, namely, construct the induced poset $Q_L$   and check whether the mapping $f$ from Theorem~\ref{birkhoff-thm} is an isomorphism. Additionally, any distributive lattice can be generated by starting with a poset $P$ and constructing $J(P)$. We utilize Theorem~\ref{birkhoff-thm} in proving that all distributive lattices have distinguishing number at most two (Theorem~\ref{distrib-lattice-thm}) and in characterizing those that have distinguishing number one (Theorem~\ref{D-one-thm}).  A proof of Theorem~\ref{birkhoff-thm} appears in \cite{St12}.

Observe that in  lattice $L_{150}$ shown in Figure~\ref{L-150-fig}, the point 75 can be written   as the join of  all the join-irreducibles less or equal to it,  namely $75=3\vee 5\vee 25$. Equivalently, in $J(Q_L)$, $\{b,c,d\} = \{b\}\cup \{c\}\cup \{c,d\}$. The next corollary shows this holds in general, that is, every point $w\in L$ can be written as the join of a unique subset of join-irreducibles of $L$. 
It  is a well-known consequence of  Birkhoff's Theorem and we  provide a proof for completeness.
 
\begin{Cor} \label{tough-cor}
Let $L$ be a distributive lattice and $f:L\rightarrow J(Q_L)$ be the isomorphism from Theorem~\ref{birkhoff-thm}.  
If $w\in L$, then $w=\bigvee _{z\in f(w)} z$. 
\end{Cor}

\begin{proof}

Fix $w \in L$ and let   $f(w) = \{y_1, y_2, \ldots, y_t\}$.  For each $i:  1 \le i \le t$, the reflexive property implies that $y_i \in f(y_i)$, and the fact that $f(w)$ is a downset implies that $f(y_i) \subseteq f(w)$.    Therefore, $f(w) = f(y_1) \cup f(y_2) \cup \cdots \cup f(y_t)$.    However, $f$ is an isomorphism, so applying $f^{-1}$ to both sides yields $w = y_1 \vee y_2 \vee \cdots \vee y_t$, as desired.

  \end{proof}

 For a distributive lattice $L$, the points of $J(Q_L)$ are downsets, and the rank of each point is the cardinality of its downset.  We record this below.
\begin{remark}
\label{rank-rem}
Every  distributive lattice $L$ is ranked and the rank of a point $w$ is $|\{z \in Q_L: z \preceq w\}|.$
\end{remark}


 \section{Distinguishing numbers}\label{sect-dist-nos}
 We begin with the definition of  the distinguishing number of a poset and some examples.
 
\begin{Def} {\rm
  A coloring of the points of poset $P$ is  \emph{distinguishing} if the only automorphism of $P$ that preserves colors is the identity. The \emph{distinguishing number} of $P$, denoted $D(P)$ is the least integer $k$ so that $P$ has a distinguishing coloring using $k$ colors.
}
\label{disting-def}
\end{Def}

Distinguishing colorings are shown in Figure~\ref{ex-fig}, and the distinguishing numbers are the following: $D(L_{pq^2})=1$, $D(L_{p^2q^2})=2$, $D(M)=1$, $D(L_{pqr})=2$, $D(S_4)=2$.    Note that while $D(M) = 1$, if we remove point $x$ from $M$, the resulting induced subposet $M-x$ has $D(M-x) = 2$;  thus an induced subposet can have a larger distinguishing number than that of the original.
Observe that antichains are the only posets for which each point must receive a different color in a distinguishing coloring.  

The \emph{comparability graph} of poset $P$ is the graph $G_P= (V,E)$ where $V$ is the ground set of $P$ and $xy \in E$ if and only if $x$ and $y$ are comparable in $P$.   Any automorphism of  a  poset $P$ is also an automorphism of its comparability graph, $G_P$.  This justifies the following remark.
 
\begin{remark}  

$D(P) \le D(G_P)$.
\end{remark}

 Some automorphisms of the graph $G_P$ are not automorphisms of the poset $P$ because they do not preserve the ordering of points in $P$.  The following example shows that $D(P)$  can differ significantly from  $D(G_P)$  and $D(\overline{G_P})$.  
If  $P$ is an  $n$-chain  then  $D(P) = 1$.  However, the  comparability graph  of $P$ is the complete graph $K_n$, and the incomparability graph of $P$ is its complement $\overline{K_n}$, and each of these has distinguishing number $n$.

\subsection{Sums of chains } 

In the next two results, we find the distinguishing number for posets consisting of the sum of chains.
 
  \begin{Prop}
    Let $P$ be the poset consisting of the sum  of $t$ chains, each consisting of $r$ points and let $k$ be the positive integer for which $(k-1)^r < t \le k^r$.  Then $D(P) = k$.  
    \label{chains-lem}
    \end{Prop}
    
    \begin{proof}
    First we find a distinguishing coloring of $P$ using $k$ colors.  There are $k^r$ different ways to color the elements of an $r$-chain when $k$ colors are available.  Coloring the elements of each $r$-chain differently is a distinguishing coloring since any automorphism of $P$ maps an $r$-chain to an $r$-chain.  Thus $D(P) \le k$.  We next show $D(P) > k-1$.    For a contradiction, suppose there is a distinguishing coloring of $P$ using $k-1$ colors.  There are $(k-1)^r$ ways to color each chain and since $t > (k-1)^r$, two chains have the same coloring.  The automorphism that swaps those two chains is non-trivial, a contradiction.
    \end{proof}
    
    Combining Proposition~\ref{chains-lem} with the following proposition, allows us to compute the distinguishing number of any poset that consists of the sum of chains.
    
    \begin{Prop} Let $P$ be the sum of chains and partition $P$ as $P_1 + P_2 + \cdots  + P_m$ where $P_i$ consists of $t_i$ chains, each consisting of $r_i$ points, where     $r_1, r_2, \cdots , r_m$ are distinct.  Then $D(P) = \max \{D(P_i): 1 \le i \le m\}$.
    \label{chains-prop}
    \end{Prop}
\begin{proof}
The result follows immediately from the fact that any automorphism of $P$ will map $P_i$ to itself for each $i$.
\end{proof}

\subsection{Distributive lattices} \label{sect-dist-latt}

We find the distinguishing number of any distributive lattice in Theorems~\ref{D-one-thm} and \ref{distrib-lattice-thm}.

 In showing that a coloring is distinguishing it can be helpful to analyze the points individually or in groups using the following concept of \emph{pinning}. 
  
\begin{Def} {\rm Let $P$ be a poset with a color assigned to each point. We say that a point $x$ is \emph{pinned} if every automorphism of $P$ that preserves colors maps $x$ to itself. }

\end{Def}
Note that  a coloring of the ground set of a poset $P$ is distinguishing precisely when every point is pinned. 

\begin{Prop}\label{tough-pin-prop}
   Let   $\phi$  be any coloring of a distributive lattice $L$.  If $\phi$ restricted to $Q_L$ pins every point of $Q_L$, then $\phi$ pins every point of $L$.
\end{Prop}

\begin{proof}
 Let $\phi$ be a coloring of $L$ so that $\phi$ restricted to $Q_L$ pins every point of $Q_L$.   By Corollary~\ref{tough-cor}, every element of $L$ is the join of a unique set of elements of $Q_L$. Since joins are preserved by isomorphism, it follows that every point of $L$ is pinned.
\end{proof}

We now have the tools to determine the distinguishing number of any distributive lattice.  Theorem~\ref{distrib-lattice-thm} shows that  any distributive lattice has distinguishing number at most two and Theorem~\ref{D-one-thm} characterizes those distributive lattices whose distinguishing number is one.   The proof of Theorem~\ref{distrib-lattice-thm} is illustrated in Examples~\ref{example-birkhoff} and \ref{example-birkhoff-2}.

 \begin{Thm} If $L$ is  a distributive lattice, then  $D(L) = 1$ if and only if $D(Q_L)=1$.
\label{D-one-thm}
\end{Thm}

\begin{proof}
By definition, any poset has distinguishing number equal to $1$ if and only if it has no non-trivial automorphisms. Let $\phi$ be a coloring of $L$  in which every vertex is colored the same. If $Q_L$  has no non-trivial automorphisms, then every point in $Q_L$ is pinned by $\phi$ and by Proposition~\ref{tough-pin-prop}, every point in $L$ is pinned. 
Conversely, if $Q_L$ has a non-trivial automorphism $\sigma$, then by Corollary~\ref{tough-cor}, $\sigma$ can be extended to a non-trivial automorphism of $L$, contradicting $D(L)=1$. 
\end{proof}


\begin{Thm}   If $L= (X, \preceq)$ is a distributive lattice, then $D(L)\leq 2$ and $D(L) = 2$ if and only if $D(Q_L)>1$. 
\label{distrib-lattice-thm}
\end{Thm}

\begin{proof}
Let $L= (X \preceq)$ be a distributive lattice and $Q_L = (Y, \prec)$ where    $Y = \{y_1,y_2, \ldots, y_t\}$.    By Theorem~\ref{birkhoff-thm}, $L$ is isomorphic to $J(Q_L)$.   We will provide a distinguishing coloring of $J(Q_L)$ using two colors, showing $D(L)\leq 2$. The remainder of the theorem follows from Theorem~\ref{D-one-thm}. 

Let $f:L \to J(Q_L)$ be the isomorphism defined in Theorem~\ref{birkhoff-thm},  and  let     $f(Y) = \{f(y_1),f(y_2), \ldots, f(y_t)\}$.  The property of being join-irreducible is preserved under isomorphism, thus   $f(Y)$ is the set of join-irreducible points of $J(Q_L)$.  

    Let $E:y_1 \prec y_2 \prec y_3 \prec \cdots \prec y_t$ be a linear extension of $Q_L$.  Color the following chain of  elements of $J(Q_L)$ using the color red: \\ $ \{f(y_1)\}, \  \{f(y_1), f(y_2)\}, \ 
\{f(y_1), f(y_2), f(y_3)\}, \  \cdots ,\ \{f(y_1), f(y_2), f(y_3), \cdots f(y_{t-1})\}.$    Color the remaining elements green.  We show this is a distinguishing coloring of $J(Q_L)$ by showing that every nontrivial  automorphism of $J(Q_L)$ preserves colors.

Since poset automorphisms preserve rank and there is at most one red vertex at each rank of $J(Q_L)$,  we know the red vertices are pinned.    Next we show all  green points in $f(Y)$ are pinned. For  $2\leq i\leq t-1$, each  $f(y_i)\in f(Y)$ is less than a unique lowest red point in the chain of red vertices of $J(Q_L)$.  In particular,   $\{f(y_i)\} \preceq \{f(y_1), f(y_2), f(y_3), \cdots , f(y_i)\}$ but $\{f(y_i)\} $ is  incomparable to all lower ranked red points. Hence each is pinned. Also, $f(y_t)$ is the only point in $f(Y)$ that is not less than any red vertex, hence it is pinned. Thus the green points in $f(Y)$ are pinned. By Corollary~\ref{tough-cor},  every point of $J(Q_L)$ that is not in $f(Y)$ is the  join of a unique set of elements of $f(Y)$, and hence is pinned.  Thus all points are pinned and the coloring is distinguishing.
\end{proof}

The next two examples illustrate the proof of Theorem~\ref{distrib-lattice-thm}.

\begin{example}
{\rm For the distributive lattice $L = L_{150}$ in Figure~\ref{L-150-fig}, the set of join-irreducible points is $Y = \{a,b,c,d\}$, where $a=2,\  b=3, \ c=5 $ and $d=25$.  For the linear extension
$E: a \prec c \prec d \prec b$ of $Q_L$, 
the chain of points in $J(Q_L)$ colored red  in the proof of Theorem~\ref{distrib-lattice-thm}
is $ a \prec ac \prec acd $ and the remaining vertices are green.   Observe that each join-irreducible point in $J(Q_L)$ except $b$ is indeed less than or equal to  a unique lowest red point:  $a \preceq a $ (rank 1), \ 
$c \preceq ac $ (rank 2), \ 
$cd \preceq acd $ (rank 3).  

}
\label{example-birkhoff}
\end{example}

\begin{example}
{\rm For the distributive lattice in Figure~\ref{fig-distrib-lat}, the set of join-irreducible points is $Y = \{a,b,c,d\}$, where $a = w_1$, $b = w_4$, $c = w_5$, and $d = w_2$.  For the linear extension 
$E: d \prec a \prec b \prec c$ of $Q_L$, 
the chain of points in $J(Q_L)$ colored red  in the proof of Theorem~\ref{distrib-lattice-thm} is $ d \prec ad \prec abd $ and the remaining vertices are green.   Observe that each  join-irreducible point of $J(Q_L)$  except $c$ is indeed less than or equal to a unique lowest red point:  $a \preceq ad $ (rank 2), \ 
$ab \preceq abd $ (rank 3), \ 
$d \preceq d $ (rank 1).  
Each point of $J(Q_L)$   is the join of a unique set of join-irreducible points of $J(Q_L)$, for example, point $acd$ is the join of $a,d,ac$.

}
\label{example-birkhoff-2}
\end{example}

Our proof of Theorem~\ref{distrib-lattice-thm}
provides a distinguishing coloring of $L$ for each linear extension of $Q_L$.  We record this in Corollary~\ref{count-coloring-cor}.

\begin{Cor}
For any distributive lattice $L$, each linear extension of $Q_L$ leads to a distinguishing coloring of $L$ using two colors, one of which appears on exactly $|Q_L| - 1$ points.
\label{count-coloring-cor}
\end{Cor}

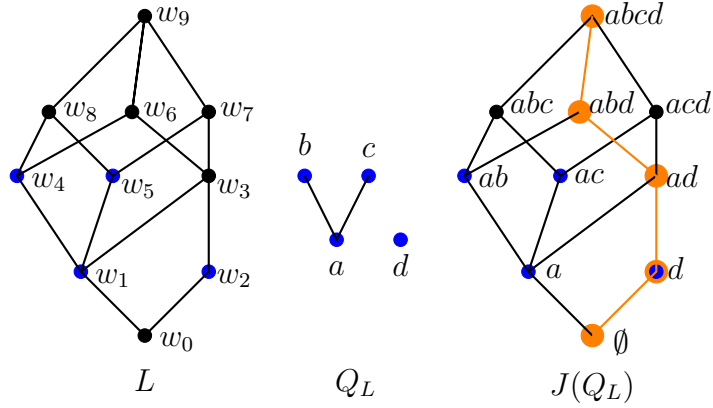
\begin{figure}[h]

\begin{center}
\begin{tikzpicture}[scale=.85]
\filldraw[black]

(-.5,2) circle [radius=3pt]
(1,3.5) circle [radius=3pt]
(2,2) circle [radius=3pt]

(1,-1.5) circle [radius=3pt]
(2,1) circle [radius=3pt]
(.8,2) circle [radius=3pt]
;
\filldraw[blue]
(0,-.5) circle [radius=3pt]
(2,-.5) circle [radius=3pt]
(-1,1) circle [radius=3pt]
(.5,1) circle [radius=3pt]
;

\draw[thick] (-.5,2)--(1,3.5)--(2,2)--(.5,1)--(-.5,2)--(-1,1) -- (0,-.5)--(.5,1);
\draw[thick] (1,-1.5)--(0,-.5)--(2,1)--(.8,2)--(1,3.5)--(.8,2)--(-1,1);
\draw[thick](1,-1.5)--(2,-.5)--(2,1)--(2,2);
\node(0) at (1, -2.2) {$L$};

\node(0) at (1.5,-1.6) {$w_0$};
\node(0) at (.5,-.6) {$w_1$};
\node(0) at (2.4,-.6) {$w_2$};
\node(0) at (-.5,.9) {$w_4$};
\node(0) at (2.4,.8) {$w_3$};
\node(0) at (.9,.8) {$w_5$};
\node(0) at (1.25,2) {$w_6$};
\node(0) at (2.45,2) {$w_7$};
\node(0) at (0,2) {$w_8$};
\node(0) at (1.45,3.5) {$w_9$};

\filldraw[blue]
(4,0) circle [radius=3pt]
(5,0) circle [radius=3pt]
(3.5,1) circle [radius=3pt]
(4.5,1) circle [radius=3pt]
;
\draw[thick] (3.5,1)--(4,0)--(4.5,1);
\node(0) at (4.3,-2.25) {$Q_L$};
\node(0) at (4,-.45) {$a$};
\node(0) at (5,-.4) {$d$};
\node(0) at (3.5,1.5) {$b$};
\node(0) at (4.5,1.4) {$c$};

\filldraw[black]
(6.5,2) circle [radius=3pt]
(8,3.5) circle [radius=3pt]
(9,2) circle [radius=3pt]

(8,-1.5) circle [radius=3pt]
(9,1) circle [radius=3pt]
(7.8,2) circle [radius=3pt]
;
\filldraw[orange]
(8,-1.5) circle [radius=5pt]
(9,-.5) circle [radius=5pt]
(9,1) circle [radius=5pt]
(7.8,2) circle [radius=5pt]
(8,3.5) circle [radius=5pt]
;
\filldraw[blue]
(7,-.5) circle [radius=3pt]
(9,-.5) circle [radius=3pt]
(6,1) circle [radius=3pt]
(7.5,1) circle [radius=3pt]
;

\draw[thick] (6.5,2)--(8,3.5)--(9,2)--(7.5,1)--(6.5,2)--(6,1) -- (7,-.5)--(7.5,1);
\draw[thick] (8,-1.5)--(7,-.5)--(9,1);
\draw[thick] (7.8,2)--(6,1);
\draw[thick] (9,1)--(9,2);
\draw[thick,orange] (8,-1.5)--(9,-.5)--(9,1)--(7.8,2)--(8,3.5);
\node(0) at (8, -2.3) {$J(Q_L)$};

\node(0) at (8.45,-1.6) {$\emptyset$};
\node(0) at (7.4,-.5) {$a$};
\node(0) at (9.3,-.45) {$d$};
\node(0) at (6.45,1) {$ab$};
\node(0) at (9.4,1) {$ad$};
\node(0) at (7.95,1) {$ac$};
\node(0) at (8.3,2.15) {$abd$};
\node(0) at (9.5,2.15) {$acd$};
\node(0) at (7.05,2.15) {$abc$};
\node(0) at (8.65,3.55) {$abcd$};

\end{tikzpicture}

\end{center}
\caption{A lattice $L$, its poset $Q_L$ of join-irreducibles, and the downset lattice $J(Q_L)$, together with a distinguishing coloring of $J(Q_L)$.}
\label{fig-distrib-lat}

\end{figure}

\subsection{Divisibility lattices}

Divisibility lattices form a subset of the set of distributive lattices. The meet of two integers is their greatest common divisor and their join is their least common multiple. For positive integer $n$, the  \emph{divisibility lattice} is the poset $L_n$ consisting of the positive integer divisors of $n$ ordered by divisibility.  Figure~\ref{ex-fig} shows the poset $L_n$ for $n=pq^2$ and $n=p^2q^2$ when $p$ and $q$ are distinct primes.  As illustrated  by this figure, the structure of $L_n$ is determined by the prime factorization of $n$.  

Let $n = p_1^{a_1} p_2^{a_2}  \cdots p_k^{a_k} $ where the $p_i$ are distinct primes and each $a_i$ is a positive integer.  
It is straightforward to check that if $f$ is an automorphism of divisibility lattice $L_n$ and $f(p_i) = p_j$ then $a_i = a_j$. The join-irreducible elements of $L_n$ are the factors of $n$ of the form $p_i^{b_i}$. 
When Theorem~\ref{distrib-lattice-thm} is translated to divisibility lattices, we can say precisely when $D(L)=1$ and when $D(L)=2$.

\begin{Thm}
Let $n$ be an integer greater than 1 and write $n = p_1^{a_1} p_2^{a_2}  \cdots p_k^{a_k} $ where $p_i$ are distinct primes and $a_i \ge 1$ for each $i$.  The divisibility lattice $L_n$ has  $D(L_n) = 1$ if the $a_i$ are distinct and $D(L_n)= 2 $ otherwise.
\label{divis-lat-thm}
\end{Thm}

When $D(L_n) = 2$, the coloring in the proof of Theorem~\ref{distrib-lattice-thm} does not always use a minimum number of red vertices, as seen in the following example.

\begin{example}{\rm
Consider the divisibility lattice $L_n$ where $n = 2310 = 2 \cdot 3 \cdot 5 \cdot 7 \cdot 11$. The join-irreducibles of $L_{2310}$ are the primes $2,3,5,7$ and 11. Each of these points has rank 1.
 Thus, in the proof of Theorem~\ref{divis-lat-thm}, the points colored red are the four in the chain 
  $2 \prec 2\cdot 3 \prec 2 \cdot 3 \cdot 5 \prec 2 \cdot 3 \cdot 5 \cdot 7 $, while the remaining vertices are colored green.

Instead, we could color the three points $5\cdot 7$, \ $ 7 \cdot 11$,\   $3 \cdot 5 \cdot 7$ red and the remaining points green.  Each of the rank 1 points is pinned as follows:  2 is pinned since it is the only rank one point not below any red point, 3 is pinned since it below a rank 3 red point and no others, 5 is pinned since it is below one rank 2 red point and one rank 3 red point,  7 is pinned since it is below all three red points, and 11 is pinned since it is below one rank 2 red point and no others.}

\end{example}

\section{Distinguishing Chromatic Number}\label{sect-dist-chr}


The distinguishing chromatic number $\chi_D(G)$ of a graph $G$  was introduced in \cite{CoTr06} and studied further by other authors, see for example \cite{BaPa17-a, BaPa17-b, CoHoTr09,ImKaPiSh17}. 
It is defined as the minimum number of colors needed to properly  color the vertices of $G$ so that the only automorphism that preserves colors is the identity.  
We next  define  an analogous parameter for posets.

\begin{Def}{\rm  A coloring of the points of poset $P$ is \emph{proper} if comparable points are assigned different colors, that is, each color class induces an antichain.  The \emph{distinguishing chromatic number} of poset $P$, denoted $\chi_D(P)$, is the least integer $k$ for which there is a coloring of $P$ that is both proper and distinguishing.
}
\end{Def}

For example, for the poset $M$ in Figure~\ref{ex-fig}, $\chi_D(M)=4$ and one proper distinguishing coloring is the following:  color 1 for $z$, color 2 for  $y$ and $w$,  color 3 for  $x$, and  and color 4 for  $v$. 

The fact that any automorphism of  a  poset $P$ is also an automorphism of its comparability graph $G_P$ justifies the following remark.
\begin{remark}  

$\chi_D(P) \le \chi_D(G_P)$.
\end{remark}
 
However,  some automorphisms of the graph  $G_P$ are not automorphisms of the poset $P$ because they do not preserve the ordering of points in $P$.    
The following result  shows that $\chi_D(P)$  can differ significantly from  $\chi_D(G_P)$. 

\begin{Prop}
 There exist posets $P$  for which the gap between $\chi_D(P)$  and   $\chi_D(G_P)$ is arbitrarily large.
\end{Prop}

\begin{proof}

 Let $j\geq 3$ be  a positive integer and $P$ be the poset consisting of  $\binom{2j}{j}$ disjoint $j$-chains. The   graph $G_P$ consists of   $\binom{2j}{j}$ copies of the graph $K_j$,   and thus  $\chi_D(G_P) = 2j$. However, $\chi_D(P)=t $, where $t$  is the smallest integer   such that $(t)_j=t(t-1)(t-2)\ldots (t-j+1)\geq \binom{2j}{j}$.  
We will show that $\chi_D(P) \le j+1$, demonstrating that the gap between $\chi_D(G_P)$ and $\chi_D(P)$ can be made arbitrarily large.

It remains to show $t \le j+1$.   For the initial value $j=3$, we have $(j+1)_j  = 4\cdot 3\cdot 2 = 24>  20 = \binom{6}{3} = \binom{2j}{j}$, so $t \le j+1$.   Observe that $$\binom{2j+2}{j+1} = \frac{2(j+1)(2j+1)}{(j+1)^2}\binom{2j}{j}<4 \binom{2j}{j}$$ whereas $(t)_{j+1} = (j+1)(t)_j$, so 
$(t)_j$ grows at a faster rate than $\binom{2j}{j}$ for $j \ge 3$.  Thus  $t \le j+1$ for $j \ge 3$ as desired.

\end{proof}
  
   \noindent
{\bf Application:}  \ The definition of $\chi_D(P)$ is related to the following problem of designing a student's course schedule. Form a poset $P$ in which the points of $P$ are the courses a student plans to take to complete a major, and $x\prec y$ if course $x$ is a prerequisite for course $y$. In a proper coloring, if two courses receive the same color, neither is a prerequisite of the other and they can be taken in the same semester. The minimum number of colors needed for a proper coloring of $P$ is the minimum number of semesters needed to complete the major. If the coloring is also distinguishing, then $P$ together with its coloring will uniquely identify the courses as well as specify which ones are taken in which semester.

The next result  can be used to determine the distinguishing chromatic number of posets consisting of the sum of chains.
 We denote the falling factorial as   $(k)_r = k(k-1)(k-2) \cdots (k-r+1)$.  
  
  \begin{Prop}
 (i) If $P$ is the poset consisting of the sum  of $t$ chains  in which each chain contains $r$ elements, and    $k$ is the positive integer for which $(k-1)_r < t \le (k)_r$, then $\chi_D(P) = k$.  
  
 (ii) Let $P$ be the sum of chains and partition $P$ as $P_1 + P_2 + \cdots  + P_m$ where $P_i$ consists of $t_i$ chains, each consisting of $r_i$ points, where     $r_1, r_2, \cdots , r_m$ are distinct.  Then $\chi_D(P) = \max \{\chi_D(P_i): 1 \le i \le m\}$.
    \label{chi-chains-prop}
    \end{Prop}
    
\begin{proof}
   Use the arguments given in the proofs of Propositions~\ref{chains-lem}  and \ref{chains-prop},
     except here the vertices of a chain must get different colors, so there are $(k)_r$ ways to color a chain of $r$ points if there are $k$ colors available.
\end{proof}

An alternate way to properly color the points of a poset is to color two points distinctly if they are \emph{incomparable}, or equivalently, so that each color class induces a chain. We call this a \emph{chain-proper} coloring.  A coloring that is both chain-proper and distinguishing
 is related to the following problem of assigning rooms to a set of scheduled events.

{\bf Application:} Represent a set of events as a poset $P$ in which the events are the points of $P$ and $x\prec y$ if event $x$ ends before event $y$ begins. In a chain-proper coloring, each color class is a set of events that can be assigned to the same room, and thus the minimum number of color classes is the number of rooms needed to schedule all of the events. If the coloring is distinguishing as well as proper, then the poset together with its coloring will uniquely identify the events as well as specifying which room each would occupy. 


As an example,  the poset $M$ in Figure~\ref{ex-fig}, requires two colors for a chain-proper coloring that is distinguishing:  color $x$, $y$, and $z$ red, and color $v$ and $w$ blue.    
   The next proposition is a lovely consequence of Dilworth's theorem.

\begin{Prop} For any poset $P$, the minimum number of colors needed for a chain-proper coloring that is also distinguishing is  the width of $P$.  
\label{chi-D-thm}
\end{Prop}

\begin{proof}
Let $k$ be the width of $P$ and let $A$ be an antichain of $P$ with $|A| =k$.  Coloring the points of $A$ properly requires $k$ colors, hence at least $k$ colors are required.    To show that $k$ colors suffice,      use Dilworth's theorem to partition the points of $P$ into $k$ sets, each of which induces a chain in $P$.  Color all points on chain $i$ using color $i$ for $i = 1,2,3 \ldots, k$.  By definition, this coloring is proper.  Observe that all chain-proper colorings are distinguishing because   each point on chain $i$ has a unique height on that chain, and height is preserved by automorphisms.  \end{proof}

\subsection{Bounds for distributive lattices}

In our next result, we again use Birkhoff's Theorem, this time to relate the distinguishing  chromatic number of a distributive lattice to that  of its poset of join-irreducibles. Note that Lemma~\ref{chi-D-bound} is tight for $L_{pq}$ when $p$ and $q$ are distinct primes.  

\begin{lemma} If $L$ is a distributive lattice, then  $\chi_D(L) \le  \chi_D(Q_L) + |Q_L|$.
\label{chi-D-bound}
\end{lemma}

\begin{proof}
First color each point of $L$ at rank $j$ using color $j$ for $0 \le j \le |Q_L|$.   This provides a proper coloring of $L$  and by Remark~\ref{rank-rem},  it uses $|Q_L| + 1$ colors.    Next, recolor the points of $Q_L$ using a proper and distinguishing coloring with $\chi_D(Q_L)$ new colors.  The resulting coloring of $L$ is still proper.  It is also distinguishing since by Corollary~\ref{tough-cor}
 any point of $L$  can be written uniquely as the join of elements of $Q_L$, and automorphisms preserve joins.     All rank 1 points of $L$ are in $Q_L$, so the color 1 is never used in the final coloring, thus we have a proper and distinguishing coloring of $L$ using $ \chi_D(Q_L) + |Q_L|$ colors.
\end{proof}

The lattice $L$ in Figure~\ref{fig-distrib-lat} has $\chi_D(Q_L) = 3$ and $|Q_L| = 4$  The proof of Lemma~\ref{chi-D-bound} provides a proper and distinguishing coloring of $L$ using 7 colors.  We can show that $\chi_D(L) \ge 6$ as follows.  At least 5 colors are needed for a proper coloring, and any proper coloring using 5 colors  assigns the same color to  $w_4$ and $w_5$, and thus is not distinguishing.  The reverse inequality, $\chi_D(L) \le 6$, follows from our next theorem, and thus lattice $L$ is an example that shows  the bound in Theorem~\ref{chi-D-bound-minus} is tight.


\begin{Thm} If $L$ is a distributive lattice and $\chi_D(Q_L)\ge 3$, then $\chi_D(L)\leq |Q_L|+ \chi_D(Q_L) - 1$. 
\label{chi-D-bound-minus}
\end{Thm}

\begin{proof} 
Let $d=\chi_D(Q_L)$ and let $\phi$ be a proper and distinguishing coloring  of 
$Q_L$, using  the colors  in the set $A = \{a_1, a_2, \ldots, a_d\}$.   Let $A_i$ be the set of points in $Q_L$ with color $a_i$ for $1 \le i \le d$, that is $A_i = \{x \in Q_L: \phi(x) = a_i\}$.   We use a different set of colors, the \emph{rank colors},  for the remaining points of $L$.   For each point $x$ in $L$, let $r(x)$ be its rank   in $L$.  As before, each uncolored point $x$ in $L$ has    $r(x) $  in the set $R =  \{0,2,3,4, \ldots, |Q_L|\}$, and we color point $x$ using color  $r(x)$.
As before, this  is a  proper and distinguishing coloring of $L$ using  $|Q_L|+d$ colors. We will 
construct a new coloring that uses one fewer color by eliminating the color 2.

We define three subsets of points of $Q_L$ as follows.

\[S_1 = \{z \in A_1:   r(z) \ge 3\}\]

\[S_2 = \{z \in A_2 :   r(z) \ge 3 \hbox { and }  \exists \  y \in A_1 \hbox{ with }  r(y) = 1 \hbox { and }  \ y \prec z\}   \]

\[S_3 = \{z \in A_3:    r(z) \ge 3 \hbox { and }  \exists \  y \in A_1 \hbox{ and } \exists \ w  \in A_2 \hbox{ with }\\  r(y)  =  r(w) = 1, \hbox { and }   \ y,w \prec z \} \]

Define a new coloring $\psi$  on points $x$ of  $Q_L$ as follows. 

\begin{equation*}
\psi(x) = 
\begin{cases} 
r(x) &  \text{if }  x\in S_1\cup S_2\cup S_3\\ 

\phi(x)  & \text{otherwise.}\\
\end{cases}
\end{equation*}

The coloring $\psi$ is proper on the points of $Q_L$ because $\phi$ gives a proper coloring of the points not in $S_1 \cup S_2 \cup S_3$,     the rank  function gives a proper coloring of the points  in $S_1 \cup S_2 \cup S_3$, and the  colors in $R$  are different from the   colors in $A$.    We next show that $\psi$ is a distinguishing coloring of $Q_L$.    Let $h$ be an automorphism of $Q_L$ that preserves the coloring $\psi$.  

First we show that $h(S_1) = S_1$, $h(S_2) = S_2$ and $h(S_3) = S_3$, that is, $h$ preserves membership in each of the sets $S_1$, $S_2$, $S_3$.  For $z \in S_1$ we have $\psi(z) = r(z)$, so $h(z)$ has a color in  the set $ R$   and thus $h(z) \in S_1 \cup S_2 \cup S_3$.  
By the definition of $S_1$, we know  $\phi(z) = a_1$, and because $\phi$ is a proper coloring,  the point $z$ is incomparable to all other points of color $a_1$. Each point in $S_2\cup S_3$ is comparable to a rank 1 point with color $a_1$. Since $h$ preserves coloring  $\psi$, we know  $h(z) \not\in S_2\cup S_3$.  Hence $h(z)\in S_1$, and $h(S_1) = S_1$. 
For $z\in S_2$, we have  $\psi(z) = r(z) $, so  $h(z)$  also has a color in  the set $ R$,   and thus $h(z) \in S_1 \cup S_2 \cup S_3$. However, $h$ is an automorphism and  $h(S_1) = S_1$, hence  $h(z) \in S_2\cup S_3$. Each point in $S_3$ is comparable to a rank 1 point with color $a_2$, and since $h$ preserves $\psi$,  we know  $h(z)\not \in S_3$. Hence $h(z)\in S_2$ and $h(S_2)=S_2$. Similarly, each point in $S_3$ has a  color in $R$ and thus $h(S_3)=S_3$ as well. 

 Since $h$ preserves $\psi$ and  $h(S_i)=S_i$, then  $h(A_i-S_i)=A_i-S_i$ and  $h(A_i)=A_i$ for $1\leq i\leq 3$. We know  $h(A_i)=A_i$ for $i\geq 4$ because $h$ preserves $\psi$. Thus, $\psi$ preserves $\phi$. 
Since $\phi$ is a distinguishing coloring of $Q_L$, we conclude that  $h$ is the identity automorphism. Thus we have shown that  $\psi$ is a distinguishing coloring of $Q_L$. 

Now we extend $\psi$ to $L$. For each pair of distinct rank 1 points $x,y$ of $Q_L$,  define 

\begin{equation*}
g(x,y) = 
\begin{cases} 
a_1 & $\mbox{if $a_1\not \in \{ \phi(x),\phi(y)\}$  }$\\
a_2 &$\mbox{if $a_1\in \{ \phi(x),\phi(y)\}$  and $a_2\not\in \{ \phi(x),\phi(y)\}$ }$\\
a_3 &\{ \phi(x),\phi(y)\} = \{a_1, a_2\}\\
\end{cases}
\end{equation*}

We now extend $\psi$ to the elements of $L$ that are not in $Q_L$.   By Remark~\ref{rank-rem},
each rank 2 point in $L$ that is not in $Q_L$ covers exactly two rank 1 points of $L$.  This allows us to define $\psi(z)$  when $r(z) = 2$ and $z \not\in Q_L$ as follows.

\begin{equation*}
\psi(z) = 
\begin{cases} 
 r(z) & z\not\in Q_L \mbox{ and } r(z)>2 \mbox{ or } r(z)=0\\
g(x,y) & z\not \in Q_L,\;r(z)=2 \mbox{ and } z \mbox { covers } x,y. \\
\end{cases}
\end{equation*}

Observe that $\psi$ uses colors from the set $\{0,3,4,\ldots, |Q_L|\} \cup A$, for a total of $d+|Q_L|-1$ colors.
Since $\psi$ is distinguishing on $Q_L$, it is distinguishing on $L$. We need additionally to show that $\psi$ is proper on $L$. The color $j$, for $j=0,3,4,\ldots, |Q_L|$, is used only on points of rank $j$, so the use of the rank colors is proper. We need to show that the set of points  of $L$ colored $a_1$ by $\psi$  form an antichain, and similarly for  the points colored $a_2$ and the points colored $a_3$.  We partition the set of $z\in L$ with $r(z) =2$ as $T_1 \cup T_2 \cup T_3$ as follows, where the rank 1 points covered by  $z$ are denoted $x$ and $y$:

(i)  $z \in T_1$ if $a_1 \not\in \{ \phi(x), \phi(y)\}$

(ii)  $z \in T_2$ if $a_1 \in \{ \phi(x), \phi(y)\}$ and $a_2 \not\in \{ \phi(x), \phi(y)\}$

(iii)  $z \in T_3$ if $a_1,a_2 \in \{ \phi(x), \phi(y)\}$  

 \smallskip

Then the $a_i$-color class of $\psi$ is $T_i\cup (A_i-S_i)$, for $1\leq i\leq 3$. 
Since $\phi$ is proper, $A_i$ is an antichain, hence $A_i-S_i$ is an antichain. Since every point in $T_i$ has rank 2, $T_i$ is an antichain. 

Suppose that $z\in T_i$  and $w\in A_i-S_i$ are comparable.  Since $w\in Q_L$, $r(w)\neq 0$. Since $z$ covers only  $x,y$ and neither is in $A_i$, then $r(w)\neq 1$. Since $r(z)=2$, and points of the same rank are not comparable, then $r(w)\neq 2$. Hence $r(w)\geq 3$ and $w\succ z$. 

Case 1: Let $i=1$. Then $\psi(z)=a_1$ and $w\in A_1-S_1$, so $\psi(w)\neq a_1$ by the definition of $S_1$. Therefore, $\psi(z)\neq \psi(w)$. 

Case 2: Let $i=2$. Then $\psi(z)=a_2$ and $w\in A_2-S_2$. Thus $w$ is not above any rank 1 point colored $a_1$, but $z\in T_2$ and $z\succ x $ and $\psi(x)=a_1$. By transitivity of the order relation, this contradicts the assumption that $w $ and $z$ are comparable. 

Case 3: Let $i=3$. Then $\psi(z)=a_3$ and $w\in A_3-S_3$. Thus, $w$ cannot be above both a rank 1 point colored $a_1$ and a rank 1 point colored $a_2$, but $z\in T_3$ and $z\succ x $, $\psi(x)=a_1$ and $z\succ y$, $\psi(y)=a_2$. By transitivity of the order relation, this contradicts the assumption that $w $ and $z$ are comparable.

Thus sets of the points colored $a_1, a_2, a_3$, respectively are antichains, and $\psi$ is proper. Since we have shown it is distinguishing, $\chi_D(L)\leq d+|Q_L|-1$, as desired.

\end{proof}

Theorem \ref{chi-D-bound-minus} is tight, as we will see in Section~\ref{Boolean}. 

\subsection{Bounds for divisibility lattices}




In the next theorem, we provide an alternative bound in Theorem~\ref{chi-D-bound-minus} in the instance when the distributive lattice is a divisibility lattice.  We  begin by coloring each point by its rank, and then recolor the join-irredicuble points using the method in  Proposition~\ref{chi-chains-prop}.
Not all join-irredicuble points need to receive new colors, so we can use fewer colors than the number needed in Proposition~\ref{chi-chains-prop}.  Recall the falling factorial function is  $(n)_k= n(n-1)(n-2)\cdots (n-k+1)$ where $(n)_0 = 1$. 

\begin{Thm} \label{better-when-smaller}
Let $n = p_1^{a_1} p_2^{a_2}  \cdots p_k^{a_k} $ where the $p_i$ are distinct primes and each $a_i$ is a positive integer and  let   $Q = Q_{L_n}$.  Partition $Q$ as $Q_1 + Q_2 + \cdots  + Q_j$ where $Q_i$ consists of $t_i$ chains, each consisting of $r_i$ points, where     $r_1, r_2, \cdots , r_j$ are distinct. 
For $1\leq i\leq k$, let $m_i$ be the smallest integer such that $t_i \leq \sum_{\ell=0}^{\min\{r_i, m_i\}}(^{r_i}_{\ell}) (m_i)_{\ell}$. Let $m=\max \{m_i: 1 \le i \le j\}$. 
Then $\chi_D(L_n) \leq  m + |Q|+1 $.  
\end{Thm}

\begin{proof} 
We begin by coloring each point in $L_n$ with its rank. There are $|Q|+1$ ranks in $L_n$. 
We then re-color some points in $Q$ with  $m$ new colors as follows. 
We can choose $\ell$ points on each $r_i$-chain to recolor in $(^{r_i}_{\ell})$ ways, and the number of ways to recolor these points with $m_i$ colors is $(m_i)_{\ell}$. Thus the total number of ways to recolor $\ell$ points  of an $r_i$-chain is $(^{r_i}_{\ell}) (m_i)_{\ell}$. Since two chains with a different number of points recolored will not have the same coloring, the total number of ways to recolor  the points of $Q_i$ using $m_i$ colors is $\sum_{\ell=0}^{\min\{r_i, m_i\}}(^{r_i}_{\ell}) (m_i)_{\ell}$.  By our choice of $m_i$, we can color the $t_i$ chains, each containing $r_i$ points,  differently.  Similarly, by our choice of $m$, there are enough colors for each value of $i$.    Chains of different lengths  in $Q$ can not map to one another under any automorphism, hence every point in $Q$ is pinned by this coloring.  By Proposition~\ref{tough-pin-prop}, every point in $L_n$ is pinned and our coloring is distinguishing.  
\end{proof}

As seen from the proof of the theorem, chains of different lengths may be considered independently. Let $Q_L$ have $t$ chains, each  of length $r$, and 
let $m$ be the smallest integer such that $t \leq \sum_{\ell=0}^{\min\{r, m\}}(^{r}_{\ell}) (m)_{\ell}$. Then the upper bound in Theorem~\ref{chi-D-bound-minus} for $\chi_D(L)$ is $\chi_D(Q_L)+|Q_L|-1$, and since an $r$-chain needs at least $r$ colors, the smallest value of $\chi_D(Q_L)$ is $r$. 
The coloring in Theorem~\ref{better-when-smaller} allows us to use fewer than $r$ new colors, and thus can be a better bound when $t$ is not too large.  For example, let $r=5$ and $t=31$, then Theorem~\ref{better-when-smaller} allows us to use only two new colors, whereas the proof in Theorem~\ref{chi-D-bound-minus} uses at least five. Thus, Theorem~\ref{chi-D-bound-minus} gives an upper bound of $5+5\cdot 31-1=159$, whereas Theorem~\ref{better-when-smaller} gives an upper bound of $2+5\cdot 31+1 = 158$. 
The formula $\sum_{\ell=0}^{\min\{r, m\}}(^{r}_{\ell}) (m)_{\ell}$ is a well-known formula for the number of ways of placing $\ell $ rooks on an $m\times r$ chessboard \cite{KaRi46}.

\subsection{Bounds for Boolean lattices} \label{Boolean}

The Boolean lattice $B_n$ is the lattice of
subsets of $\{1,2,3, \ldots, n\}$, ordered by
inclusion. It is a distributive lattice and its
join-irreducibles are the singletons  $\{ \{k\}\;|\;1\leq k\leq
n\}$.  The number of points in any longest chain of $B_n$ is $n+1$.

By Theorem~\ref{chi-D-bound-minus}, $\chi_D(B_n)\leq n+n=2n$. The next theorem has a tighter bound. 

\begin{Thm} \label{boolean-th} Let $B_n$ be the Boolean lattice of $\{1,2,3,\ldots, n\}$. Then
$\chi_D(B_n)\leq n+3$. 
\end{Thm}

\begin{proof} Let $n$ be odd. We initially color each element $x$ with $r(x)$, for $0\leq r(x)\leq n$. This coloring is proper. Next we will recolor some of the vertices to obtain a distinguishing coloring, using two new colors, $a$ and $b$. 

Let $S_i = \{i, i+1, i+2, \ldots, 2i-1\}$ for $1 \le i \le \frac{n+1}{2}$.  For $1\le i < j \le \frac{n+1}{2}$,  we have $i \in S_i - S_j$ and $2j-1 \in S_j - S_i$, thus the $S_i$ form an antichain. We change the color of each of the $S_i$ to $a$.  
  
Similarly, let $T_i = \{n-(i-1), n-i, n-(i+1), \ldots,   n-2(i-1)\}$ for $1 \le i \le \frac{n+1}{2}$. For 
$1\le i < j \le \frac{n+1}{2}$,  we have $n-(i-1) \in T_i - T_j$ and $n-2(j-1) \in T_j - T_i$, thus the $T_i$ form an antichain. We change the color of each of the $T_i$ to $b$.

For example, let $n=7$. Then we color $\{1\}, \{2,3\}, \{3,4,5\}, \{4,5,6,7\}$ with $a$, and 
we color $\{7\}, \{6,5\}, \{5,4,3\},\{4, 3, 2, 1\}$ with $b$
An alternative description of the red elements
is, for each $1\leq i\leq \frac{n+1}{2}$,  $S_i 
$ contains the $i$ smallest elements greater
than or equal to $i$. Similarly, the blue
elements are described as, for each $1\leq j\leq 
\frac{n+1}{2}$, $T_j$ contains the $j$ largest
elements less than or equal to $n-j$.

We show that the coloring is distinguishing. Each point of $B_n$ colored $a$ is pinned because it is the only point colored $a$ in its rank. Similarly, each point of $B_n$ colored $b$ is pinned. Since $Q_{B_n}$ is the set of rank 1 points, it is enough to show that the rank 1 points are pinned.  
Given $i$, $1\leq i\leq \frac{n+1}{2}$, then the
highest ranked point colored $a$ that contains $i$ is
$S_i$. For $\frac{n+1}{2}\leq j$, the highest
ranked point colored $a$ that contains $j$ is
$S_{\frac{n+1}{2}}$. Thus $\{i\}$ is  pinned for
$1\leq i\leq \frac{n-1}{2}$.

Similarly, given $j$, $\frac{n+1}{2}\leq j\leq
n$, the lowest ranked point colored $b$ that contains
$j$ is $T_j$. The lowest ranked point colored $b$
that contains $i$ is $T_{\frac{n+1}{2}}$, for 
$1\leq i\leq \frac{n+1}{2}$. Thus $\{j\}$ is
pinned for $\frac{n+3}{2}\leq j\leq n$. Now
$\frac{n+1}{2} $ is the only element that is contained in both a point colored $a$ and a point colored $b$
 at rank $\frac{n+1}{2}$. Thus, $\{ \frac{n+1}{2}\}$ is
pinned. Hence all the join-irreducibles of $B_n$
are pinned, and this coloring is
distinguishing.

Let $n$ be even. Then color all the elements by
their rank and then alter the coloring using the
same red and blue elements as in the case for
$B_{n-1}$. In this coloring, $n$ will not appear
in any  point colored $a$ or $b$, but each of
$1,2,\ldots n-1$ will do so. Thus $\{n\}$ is
pinned. The other join-irreducibles of $B_n$ are
pinned for the same reasons as in the previous
argument.
\end{proof}

The examples in the next two propositions  show that the bound in Theorem~\ref{chi-D-bound-minus} is tight when $\chi_D(Q_L) = 3$ or $4$.

\begin{Prop} The Boolean Lattice $B_3$ has $\chi_D(B_3)= 5$.
\end{Prop}
\begin{proof}
Observe that $|Q_{B_3}| = 3$ and $\chi_D(Q_{B_3}) = 3$, so by Theorem~\ref{boolean-th}, $\chi_D(B_3)\leq 5$.    The following is a proper and distinguishing coloring of $B_3$ using five colors:  $\emptyset $ is red, $\{1\}$ and $\{2,3\}$ are blue,  $\{3\}$ and $\{1,2\}$ are green, $\{2\}$ and $\{1,3\}$ are yellow, and $\{1,2,3\}$ is purple.
 Thus,  $\chi_D(B_3)= 5$.
\end{proof}

\begin{Prop}
The Boolean lattice $B_4$ has 
$\chi_D(B_4)= 7$.
\end{Prop}
\begin{proof} 
Observe that $|Q_{B_4}| = 4$ and $\chi_D(Q_{B_4}) = 4$, so by Theorem~\ref{boolean-th}, $\chi_D(B_4)\leq 7$.  
Thus we must  show that $\chi_D(B_4) >6$.  Suppose for a contradiction that we have a proper and distinguishing coloring  $\phi$ of $B_4$ using 6 colors.  One color is used for $\emptyset$ and another for $\{1,2,3,4\}$, so the remaining points must be colored using four colors:  red, blue, yellow and green.  We consider cases depending on the number of colors used on the rank 1 points.

\noindent {\bf Case 1: }  Four colors are used on the rank 1 points.

Without loss of generality we may assume $\{1\}$ is red, $\{2\}$ is blue, $\{3\}$ is green, and  $\{4\}$ is yellow.  Since $\phi$ is proper, we know $\{1,2,3\}$ must be yellow and $\{1,2,4\}$ must be green.    Now no color is available for $\{1,2\}$ since it is comparable to points that use all four colors.

\noindent {\bf Case 2:}  Three colors are used on the rank 1 points.

Without loss of generality we may assume $\{1\}$ is red, $\{2\}$ is blue,  and $\{3\}$ and  $\{4\}$ are yellow.   Since $\phi$ is proper, the colors are forced on all points except $\{3,4\}$, namely, $\{1,2,3\}$  and $\{1,2,4\}$ are green, $\{1,2\}$ is yellow, $\{2,3\}$ is red, $\{1,3\}$ is blue,  $\{2,4\}$  is red, $\{1,3,4\}$  and $\{2,3,4\}$ are green,  and $\{1,4\}$ is blue.  Regardless of whether $\{3,4\}$ is red or blue, the automorphism that swaps all instances of $3$ and $4$ preserves colors, contradicting our assumption that $\phi$ is distinguishing.

\noindent {\bf Case 3:}  Two colors are used on the rank 1 points.

First consider the instance that the two colors each appear on two rank 1 points.
Without loss of generality we may assume $\{1\}$ and  $\{2\}$ are red, $\{3\}$ and  $\{4\}$ are blue, and $\{1,3\}$ is green.  Since $\phi$ is proper, the colors are forced on all points except $\{1,2\}$ and $\{3,4\}$ as follows:  $\{1,2,3\}$ and $\{1,3,4\}$ are yellow, $\{2,3\}$ is green, $\{1,4\}$ is green,  $\{1,2,4\}$  and $\{2,3,4\}$ are yellow, and $\{2,4\}$ is green.  Regardless of the colors of $\{1,2\}$ and $\{3,4\}$, the automorphism that swaps 1 and 2 and also swaps 3 and 4 preserves colors, contradicting our assumption that $\phi$ is distinguishing.

Now consider the instance that one color appears on three of the rank 1 points and the other appears on one.  Without loss of generality we may assume $\{1\}$ and  $\{2\}$  and  $\{3\}$ are red,   $\{4\}$ is blue, and $\{3,4\}$ is green.  Since $\phi$ is proper, the colors are forced on all points except $\{1,2\}$, $\{1,3\}$,  $\{2,3\}$, and $\{1,2,3\}$ as follows:   $\{1,3,4\}$ and $\{2,3,4\}$ are yellow, $\{2,4\}$ and $\{1,4\}$ are green, and $\{1,2,4\}$ is yellow.  The remaining points  $\{1,2\}$, $\{1,3\}$,  $\{2,3\}$ must be green or blue, and two of  them must be the same color.  Without loss of generality, $\{1,2\}$ and  $\{1,3\}$ are the same color, but then the automorphism that swaps 2 and 3 preserves colors, contradicting our assumption that $\phi$ is distinguishing.

\noindent {\bf Case 4:}  The rank 1 points use one color.

If the rank 3 points use two or more colors, then  we reach a contradiction using  previous cases since $\phi$ is also a proper and distinguishing coloring of the dual of $B_4$.  Thus without loss of generality we may assume $\{1\}$,  $\{2\}$,  $\{3\}$ and  $\{4\}$ are red, and $\{1,2,3\}$, $\{1,2,4\}$, $\{1,3,4\}$, and$\{2,3,4\}$ are blue, and the rank 2 points are green and yellow.    Any such  coloring is proper but not distinguishing.  Coloring the rank 2 points using green and yellow is equivalent to giving a distinguishing coloring to the edges of the $K_4$ graph using green and yellow.  We  show this is impossible. 

   If such a coloring were possible, without loss of generality,  at least three edges are green.  First suppose there is a triangle of green edges, say $\{1,2\}$, $\{1,3\}$, and $\{2,3\}$ are green.  If the remaining edges are yellow, the automorphism $(123)(4)$ preserves colors.  If one additional edge is green, say $\{1,4\}$, then the automorphism $(23)(1)(4)$ preserves colors.  If two additional edges are green, say $\{1,4\}$  and $\{2,4\}$  then the automorphism $(12)(3)(4)$ preserves colors.  If there is no triangle of green edges, then  without loss of generality, $\{1,2\}$, $\{2,3\}$ and $\{3,4\}$ are green and the remaining three edges are yellow.  In this instance, the automorphism $(14)(23)$ preserves colors.
\end{proof}


\section{Rank-Connected Planar Posets} \label{sect-planar}

 In this final section of results, we consider rank-connected planar lattices.
 
In Section~\ref{sect-poset-defn} we defined planar posets and ranked posets.  Note that a planar poset is a lattice if it has both a minimal and maximal element. 

\begin{Def} {\rm A ranked poset is \emph{rank-connected} if every pair of consecutive ranks, considered as a vertex-induced subgraph is connected. }\end{Def}

\begin{Def}{\rm  Incomparable points $x$ and $y$ are \emph{twins}  if they have the same relationship to all other points of the poset.  A poset is \emph{twin-free} if it has no twins.  }
\end{Def}

Cohen-Macaulay posets, are a well-known class of posets in the study of flag $f$-vectors of simplicial complexes, for example, see  \cite{St96}. Although it is difficult to obtain a complete characterization of the set of flag $f$-vectors of Cohen-Macaulay posets, many subclasses of this set are lexicographically shellable and thus have an explicit shelling. Collins \cite{Co92} has shown that for ranked, planar lattices, being Cohen-Macaulay is the same as being rank-connected. In this section, we show that the distinguishing number of a rank-connected, twin-free, planar lattice is less than or equal to 2. The proof is completely different from the proof of Theorem~\ref{distrib-lattice-thm}, which uses the regular structure of a distributive lattice. However, the 2-coloring of the points is again a chain in the poset.

We say that a Hasse diagram for  ranked planar poset is a \emph{standard diagram} if  it is planar, all points at a given rank have the same $y$-coordinate, and all edges are straight line segments.  
The following result appears to be part of the folklore of the field \cite{Tr19}.  We include a proof for completeness.

\begin{Prop}
If $P$  is a ranked  planar poset with $\hat{0}$ and $ \hat{1}$ then $P$ has a standard~diagram.
\label{folklore-prop}
\end{Prop}

\begin{proof}

Partition the points of $P$ by rank so that $R_i$  is the set of points of rank $i$ for $0 \le i \le n$.    Since $P$ is ranked, each covering edge in $P$ is between points at consecutive ranks.  Suppose we have a planar Hasse diagram for $P$ in which the points of each rank have the same $y$-coordinate for ranks $k$ and lower and every edge between points of rank at most $k$  is a straight line segment.  If $k = n$ we are done, so assume $k < n$.    

Let $w_1, w_2, w_3, \ldots, w_r$ be the elements of $R_k$ listed from left to right in the Hasse diagram.  If $n = k+1$ then $R_{k+1} = \{\hat{1}\}$ and we can draw a straight line segment from each element of $R_k$ to $\hat{1}$, and this completes the proof.    Otherwise, $n \ge k+2$.

Let $L_k$ be the horizontal line containing the points of $R_k$ and $L_{k+1}$ the horizontal line containing the point(s) of $R_{k+1}$ with lowest $y$-coordinate.  In Figure~\ref{fig-planar}, the only point   of $R_{k+1}$ with lowest $y$-coordinate is $z_2$.   Order the edges between $R_k$ and $R_{k+1}$ from left to right as $e_1, e_2, \ldots, e_t$ by their order in the strip between lines $L_k$ and $L_{k+1}$.  This order is uniquely determined because the diagram is planar. This ordering of edges induces an ordering of the points in $R_{k+1}$ as follows: for any $u,v\in R_{k+1}$, we order $u$ before $v$ if the leftmost edge $e_i$ incident to $u$ is to the left of the leftmost edge $e_j$ incident to $v$. The resulting order is $z_1, z_2, z_3, \ldots , z_s$ and this is illustrated  by the example in Figure~\ref{fig-planar}.  

For each $z_i$ that is above line $L_{k+1}$, choose an edge $e_j$ incident to $z_i$ and relocate $z_i$ to the point $z_i'$ where edge $e_j$ meets line $L_{k+1}$.  In Figure~\ref{fig-planar},  in each case   the first such edge was selected.  For each $z_i \in R_{k+1}$, let $N(z_i)$ be the set of points in $R_k$ covered by $z_i$.  The points in $N(z_i)$  have consecutive indices, for otherwise, there would be a point in $R_k$ with no upward route to $\hat{1}$.   Thus we can think of $z_i$ and its edges to $N(z_i)$ as forming a cone.  If   $N(z_i) = \{w_r, w_{r+1}, \ldots, w_t\}$, then $z_i$ is the only member of $R_{k+1}$ that covers any of the internal points $w_{r+1}, w_{j+2}, \ldots, w_{t-1}$, for any other such element in $R_{k+1}$ would have no upward path to $\hat{1}$.  Thus cones intersect only at their outermost points.  Indeed, if $w_r \neq w_t$, then no other $z_j$ of $R_{k+1}$ can cover \emph{both} $w_r$ and $w_t$, because this would imply that one of $z_i, z_j$ would have no upward path to $\hat{1}$.  Thus  for all   $z_i, z_j\in R_{k+1}$ with $i \neq j$ we have,  $|N(z_i)\cap N(z_j)|\leq 1$.  Furthermore,    if $i<j$ and $N(z_i)\cap N(z_j)\neq \emptyset$, then by the planarity assumption and the way we indexed the $z_i$'s, the unique point in $N(z_i)\cap N(z_j)$ is the rightmost element of $R_k$ incident to $z_i$ and also is  the leftmost element of $R_k$ incident to $z_j$. 
Starting with edges incident to $z_1 $ and continuing rightward, we can draw the edges between $R_{k+1}$ and $R_k$   as straight line segments from the points $z_1', z_2', \ldots, z_s'$ on $L_{k+1}$ to the points $w_1, w_2, \ldots w_r$ on $L_k$ and by the properties above we know that none of these segments  cross.  This is illustrated in the right portion of Figure~\ref{fig-planar}.

It remains to reroute the edges between $R_{k+1}$ and $R_{k+2}$.  For each $z_i' \in R_{k+1}$, form a narrow band from $z_i'$ to $z_i$ along the original edge $e_j$ that is incident to $z_i'$ (see Figure~\ref{fig-planar}).  Each edge between $z_i$ and $ R_{k+2}$ will now start at $z_i'$, travel through this band to $z_i$ and continue on its original path to its destination in $R_{k+2}$.  The result is a planar Hasse diagram for $P$ in which points at each rank are each located on a horizontal line for ranks $k+1$ and lower, and edges between points  of ranks at most  $k+1$ are straight line segments.  The result follows by induction.
\end{proof}

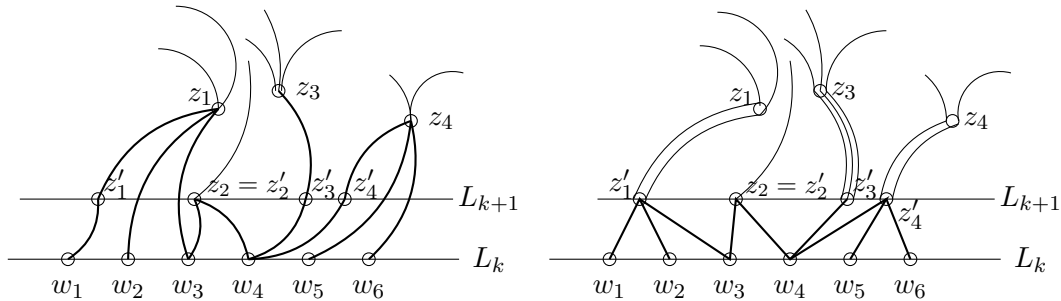
\begin{figure}[h]

\begin{center}
\begin{tikzpicture}[scale=.8]
\draw
(0,0) circle [radius=3pt]
(1,0) circle [radius=3pt]
(2,0) circle [radius=3pt]
(3,0) circle [radius=3pt]
(4,0) circle [radius=3pt]
(5,0) circle [radius=3pt]
(.5,1) circle [radius=3pt]
(2.1,1) circle [radius=3pt]
(3.95,1) circle [radius=3pt]
(4.6,1) circle [radius=3pt]
(2.5,2.5) circle [radius=3pt]
(3.5,2.8) circle [radius=3pt]
(5.7,2.3) circle [radius=3pt]
;

\draw (-1,0)--(6.5,0);
\draw (-.8,1)--(6.4,1);
\draw[thick] (.5,1) to [bend left] (2.5,2.5);
\draw[thick] (.5,1) to [bend left] (0,0);
\draw[thick]  (2.5,2.5) to [bend right] (1,0);
\draw[thick]  (2.5,2.5) to [bend right] (2,0);
\draw[thick] (2.1,1) to [bend left] (2,0);
\draw[thick] (2.1,1) to [bend left] (3,0);
\draw[thick] (3,0) to [bend right]  (3.95,1);
\draw[thick] (3.95,1) to [bend right]  (3.5,2.8);
\draw[thick] (3,0) to [bend right] (4.6,1) ;
\draw[thick] (4.6,1) to [bend left] (5.7,2.3) ;
\draw[thick] (4,0) to [bend right] (5.7,2.3) ;
\draw[thick] (5,0) to [bend right] (5.7,2.3) ;
\draw (2.5,2.5) arc (0:90:1);
\draw (2.5,2.5) arc (-45:90:1);

\draw(2.1,1) to [bend right] (3,3.3);

\draw (3.45,2.8) arc (0:60:1);
\draw (3.55,2.8) arc (180:90:1);
\draw (3.5,2.8) arc (-10:30:2);

\draw (5.7,2.3) arc (0:50:1);
\draw (5.7,2.3) arc (190:75:.7);


\node(0) at (7,0) {\small $L_k$};
\node(0) at (7,1) {\small $L_{k+1}$};
\node(0) at (0,-.5) {\small $w_1$};
\node(0) at (1,-.5) {\small $w_2$};
\node(0) at (2,-.5) {\small $w_3$};
\node(0) at (3,-.5) {\small $w_4$};
\node(0) at (4,-.5) {\small $w_5$};
\node(0) at (5,-.5) {\small $w_6$};
\node(0) at (.8,1.25) {\small $z'_1$};
\node(0) at (3,1.23) {\footnotesize $z_2=z'_2$};
\node(0) at (4.25,1.25) {\small $z'_3$};
\node(0) at (4.95,1.25) {\small $z'_4$};
\node(0) at (2.2,2.7) {\small $z_1$};
\node(0) at (4,2.8) {\small $z_3$};
\node(0) at (6.2,2.3) {\small $z_4$};

\draw
(9,0) circle [radius=3pt]
(10,0) circle [radius=3pt]
(11,0) circle [radius=3pt]
(12,0) circle [radius=3pt]
(13,0) circle [radius=3pt]
(14,0) circle [radius=3pt]
(9.5,1) circle [radius=3pt]
(11.1,1) circle [radius=3pt]
(12.95,1) circle [radius=3pt]
(13.6,1) circle [radius=3pt]
(11.5,2.5) circle [radius=3pt]
(12.5,2.8) circle [radius=3pt]
(14.7,2.3) circle [radius=3pt]
;

\draw (8,0)--(15.5,0);
\draw (8.8,1)--(15.4,1);
\draw[thick] (9,0)--(9.5,1)--(10,0);
\draw[thick] (9.5,1)--(11,0)--(11.1,1)--(12,0)--(12.95,1);
\draw[thick] (12,0)--(13.6,1)--(13,0);
\draw[thick] (13.6,1)--(14,0);
\draw (11.5,2.55) arc (0:90:1);
\draw (11.5,2.4) arc (-45:90:1);

\draw(11.1,1) to [bend right] (12,3.3);

\draw (12.4,2.8) arc (0:60:1);
\draw (12.6,2.8) arc (180:90:1);
\draw (12.5,2.8) arc (-10:30:2);

\draw (14.6,2.3) arc (0:50:1);
\draw (14.8,2.3) arc (190:75:.7);

\draw (13.7,1) to [bend left] (14.8,2.25);
\draw (13.5,1) to [bend left] (14.6,2.3);
\draw (12.85,1) to [bend right] (12.4, 2.8);
\draw (12.95,1) to [bend right] (12.5, 2.8);
\draw (13.05,1) to [bend right] (12.6, 2.8);
\draw(9.4,1) to [bend left] (11.5, 2.6); 
\draw(9.6,1) to [bend left] (11.5, 2.4);

\node(0) at (16,0) {\small $L_k$};
\node(0) at (16,1) {\small $L_{k+1}$};
\node(0) at (9,-.5) {\small $w_1$};
\node(0) at (10,-.5) {\small $w_2$};
\node(0) at (11,-.5) {\small $w_3$};
\node(0) at (12,-.5) {\small $w_4$};
\node(0) at (13,-.5) {\small $w_5$};
\node(0) at (14,-.5) {\small $w_6$};
\node(0) at (9.2,1.25) {\small $z'_1$};
\node(0) at (11.95,1.22) {\footnotesize $z_2=z'_2$};
\node(0) at (13.25,1.25) {\small $z'_3$};
\node(0) at (14,.75) {\small $z'_4$};
\node(0) at (11.2,2.7) {\small $z_1$};
\node(0) at (12.9,2.8) {\small $z_3$};
\node(0) at (15.1,2.3) {\small $z_4$};

\end{tikzpicture}

\end{center}

\caption{An illustration of the induction step in the proof of Proposition~\ref{folklore-prop}.}

\label{fig-planar}

\end{figure}

\begin{remark}
Using Proposition~\ref{folklore-prop},
it is not hard to show that the poset $L_{pqr}$ in Figure~\ref{ex-fig} is not planar.
\label{B-not-planar}
\end{remark}

The next lemma appears in \cite{Co92} and is helpful to us in the induction proof of Theorem~\ref{twin-free}. 
\begin{lemma}
\label{karen-lem}
Given a standard diagram of a planar rank-connected poset that has a $\hat{0}$ and a $\hat{1}$, there exists a rank in which the leftmost element $x$ covers exactly one element, $a$, and is covered by exactly one element, $b$.  Furthermore, if there is an element immediately to the right of $x$, it also covers $a$ and is covered by $b$.
\end{lemma}

\begin{Thm}   If $P$ is a twin-free, rank-connected planar lattice, then $D(P) \le 2$. \label{twin-free}
\end{Thm}

\begin{proof}
Let $r$ be the maximum rank of  points in $P$.  In any automorphism of $P$, rank is preserved.  If there exists a point $x$ in $P$ (other than $\hat{0}$ and $\hat{1}$) for which $x$ is the only element of its rank, then $x$ is pinned.    Thus we need only pin the elements below $x$ and separately the elements above $x$.  So without loss of generality, we may assume $P$ has at least two points at each rank other the lowest and highest ranks.   

By Proposition~\ref{folklore-prop}, we may fix a standard diagram of $P$.   At each rank, there is a leftmost point in the diagram, and  because the diagram is planar and  $P$ is rank-connected, the union of these points forms maximal chain $C_0$ from $\hat{0}$ to $\hat{1}$.
  Color the points on chain $C_0$ red, except for its minimal and maximal elements.  Since there is at most one red point at each rank, these points are pinned.  Color the remaining points  blue.  We show this coloring is distinguishing.

We apply  Lemma~\ref{karen-lem},  repeatedly to obtain a sequence of chains
 $C_0, C_1, \ldots, C_n$, each from $\hat{0}$ to $\hat{1}$, so that $C_i$ and $C_{i+1}$ are identical except for two points $x_i \in C_i$ and $x_{i+1} \in C_{i+1}$ where $x_i$ and $x_{i+1}$ have the same rank in $P$ and $x_{i+1}$ is immediately to the right of $x_i$ in the diagram.  
 We know that $x_0$ is pinned since it is red and we proceed by induction.
 
Assume the points $x_0, x_1,  \ldots, x_{j-1}$ are pinned, thus the points in $C_{j-1}$ are pinned.  If there are no  points at $x_j$'s rank that lie to the right of $x_j$ then $x_j$ is pinned since all remaining points at that rank are already pinned.  Otherwise, there exists one or more points at $x_j$'s rank that lie to the right of  $x_j$.  

Let $a$ be the point immediately below $x_j$ on chain $C_j$ and $b$ the point immediately above $x_j$ on $C_j$.   Suppose  $x_j$ is not pinned, so thus there exists a nontrivial automorphism $\phi$  of $P$  with  $\phi(x_j) = w_j$ for some $w_j \neq x_j$.  We know $w_j$ has the same rank as $x_j$ and is located to the right of $x_j$ since the points to the left of $x_j$ are already pinned.    Since $\phi$ is an automorphism, $\phi(a) \prec \phi(x_j) \prec \phi(b)$ and since $a$ and $b$ are pinned we have $a \prec w_j \prec b$.  
Partition the set of points with $b$'s rank as $B_1 \cup B_2 \cup \{b\}$  where the points in $B_1$ lie to the left of $b$ and the points in $B_2$ lie to the right of $b$.   By planarity, $w_j$ is not adjacent to any point in $B_1$.  However,   the points in $B_1 \cup \{b\}$ are pinned by our induction hypothesis, and $\phi(x_j) = w_j$, so $x_j$ cannot be adjacent to any point of $B_1$ either.    Also by planarity, $x_j$ is not adjacent to any points in $B_2$, thus the only point at $b$'s rank that is adjacent to $x_j$ is the point $b$.  Similarly, the only point at $a$'s rank adjacent to $x_j$ is $a$.  So $x_j$ is adjacent only to $a$ and $b$.  The same must be true of $w_j$ since $\phi(x_j) = w_j$ and $a$ and $b$ are pinned. 
  This means $x_j$ and $w_j$ are twins in $P$, a contradiction.  Thus $x_j$ is pinned and this completes the induction.
\end{proof}

\section{Open Questions}

We conclude with some open questions.

\begin{Ques} \rm Is Theorem~\ref{boolean-th} tight for all $n\geq 5$?
\end{Ques}

\begin{Ques} \rm Many theorems about 2-distinguishability can be proven using the Motion Lemma, proved by Russell and Sundaram \cite{RuSu98}.  
Is there a proof of Theorem~\ref{distrib-lattice-thm} using the Motion Lemma?
\end{Ques}

\begin{Ques} \rm Hadjicostas \cite{Ha19} has found the generating function for the number of distinguishing 2-colorings of an $n$-cycle. Given a distributive lattice $L$, what is the generating function of the number of distinguishing 2-colorings of $L$?
\end{Ques}

\begin{Ques} \rm The {\it cost} of 2-distinguishing a graph $G$ is the minimum size of a color class in any 2-distinguishing labeling of $G$, see \cite{Bo13}. What is the cost of 2-distinguishing a distributive lattice? The cost may be smaller than the minimum sizes of color classes in the 2-labelings we have used in our proofs. For example, the cost of 2-distinguishing $L_{p^2q^2}$ is 1, as can be seen by coloring point $p$ red and the remaining points blue, whereas the cost of 2-distinguishing $L_{pqr}$ is 2, because there will still be an automorphism that preserves the colors even if one point is fixed. 
\end{Ques}

\end{document}